\documentclass[a4paper,12pt]{article}
\usepackage{amsmath,amssymb} \usepackage{euler,eucal}
  \usepackage[utf8]{inputenc}
  \usepackage[T1]{fontenc}
  \usepackage[english]{babel}
  \usepackage{tgpagella}
\usepackage[unicode=true,plainpages=false]{hyperref}
\usepackage{graphicx}
\hypersetup{colorlinks=false,pdfborder={0 0 0.1},linkcolor=magenta,anchorcolor=magenta,urlcolor=blue,citecolor=blue,
          pdftitle={On the decay of elements of inverse triangular Toeplitz matrix},
          pdfauthor={Neville Ford, Dmitry V. Savostyanov, Nickolay L. Zamarashkin},
          pdfkeywords={triangular Toeplitz matrix, Volterra-type equation,  slow decay,  stability}
          }
\newcommand*{\doi}[1]{doi: \href{http://dx.doi.org/#1}{#1}}
\usepackage[left=25mm,right=25mm,top=25mm,bottom=30mm,a4paper]{geometry}
\usepackage{algorithmic}
\usepackage[ruled]{algorithm}
\usepackage{amsthm}
\theoremstyle{definition}

\newtheorem{remark}{Remark}
\newtheorem{example}{Example}
\newtheorem{lemma}{Lemma}
\newtheorem{statement}{Statement}
\newtheorem{corollary}{Corollary}
\newtheorem{theorem}{Theorem}
\newtheorem{definition}{Definition}
\usepackage{tikz} \usetikzlibrary{positioning,shapes,arrows}
\def\eps{\varepsilon}
\def\phi{\varphi}

\def\l{l}

\def\x{\mathbin{\times}}
\def\o{\mathbin{\raise1pt\hbox{$\scriptscriptstyle\mathord\otimes$}}}
\def\had{\mathbin{\raise1pt\hbox{$\scriptstyle\mathord\odot$}}}

\begin{document}
\author{
    Neville Ford\footnotemark[2], 
    D. V. Savostyanov\footnotemark[2]~\footnotemark[3],
    N. L. Zamarashkin\footnotemark[3] }
\title{On the decay of elements of inverse triangular Toeplitz matrix%
\thanks{During this work, D. V. Savostyanov was supported by the Leverhulme Trust as a Visiting Research Fellow at the University of Chester.
         }}
\date{August 03, 2013}
\maketitle
\renewcommand{\thefootnote}{\fnsymbol{footnote}}
\footnotetext[2]{University of Chester, Parkgate Road, Chester, CH1 4BJ, UK ({\tt njford@chester.ac.uk})}
\footnotetext[3]{Institute of Numerical Mathematics of Russian Academy of Sciences, Gubkina 8, Moscow 119333, Russia ({\tt [dmitry.savostyanov,nikolai.zamarashkin]@gmail.com})}
\renewcommand{\thefootnote}{\arabic{footnote}}

\begin{abstract}
We consider half-infinite triangular Toeplitz matrices with slow decay of the elements and prove under a monotonicity condition that elements of the inverse matrix, as well as elements of the fundamental matrix, decay to zero.
We also provide a quantitative description of the decay of the fundamental matrix in terms of $p$--norms.
Finally, we prove that for matrices with slow log--convex decay the inverse matrix has fast decay, i.e. is bounded.
The results are compared with the classical results of Jaffard and Veccio and illustrated by numerical example.
\par {\it AMS:}
15A09, % Matrix inversion
15B05, % Toeplitz, Cauchy and related matrices
45E10, % Integral equations of the convolutional type
\end{abstract}

\section{Introduction}
Consider a half--infinite triangular Toeplitz matrix, defined by a sequence $a=\{a_k\}_{k=0}^\infty$ as follows
\begin{equation}\nonumber
 A = \begin{bmatrix} 
       a_0 & & & & \\
       a_1 & a_0 & & & \\
       a_2 & a_1 & a_0 & & \\
       a_3 & a_2 & a_1 & a_0 & \\
       \cdot & \cdot & \cdot & \cdot & \cdot  \\
 %      a_{n-1} & \ldots & \ldots & \ldots & a_0
     \end{bmatrix}.
\end{equation}
It is easy to see that if $a_0\neq0,$ the matrix $A$ is invertible.
The inverse matrix $B=A^{-1}$ is also triangular Toeplitz with elements $b=\{b_k\}_{k=0}^\infty$ given by the following formula
\begin{equation}\nonumber% \label{eq:inv}
b_0 = \frac{1}{a_0}, \qquad b_k = - \frac{1}{a_0} \sum_{j=0}^{k-1} a_{k-j} b_j, \qquad\mbox{for}\quad k \geq 1. 
\end{equation}
Since $A$ and $B$ are triangular, the inverse of the $k\times k$ leading submatrix of $A$ is the $k \times k$ leading submatrix of $B.$
In this paper we consider matrices $A$ with non-negative elements $a_k\geq0,$ assuming without loss of generality (w.l.o.g.) that $a_0=1.$ 
We are motivated by the convolutional Volterra equation of the first kind with Abel--type kernel, for which $a_k \sim (k+1)^{-\alpha}.$ 
With this example in mind we study the asymptotic properties of sequences $a=\{a_k\}$ and $b=\{b_k\}.$

From an asymptotic point of view, the following three cases can be considered:
\begin{enumerate}
 \item \emph{fast decay}, i.e., $\sum_{k=0}^\infty |a_k| < \infty;$ 
 \item \emph{slow decay}, i.e., $a_k \to 0,$ $\sum_{k=0}^\infty |a_k| = \infty;$
 \item \emph{stagnation}, i.e., $a_k \to a_* > 0.$
\end{enumerate}
The first case includes matrices with superlinear decay, i.e., $a_k < c (1+k)^{-\alpha}$ for some $\alpha>1$ and $c>1.$ 
They were considered by Jaffard~\cite{jaffard-1990} in a very general framework of matrices with Toeplitz-type spatial decay.
The classical result of Jaffard shows that if the inverse matrix $B=A^{-1}$ is bounded, then it has the same polynomial decay of coefficients as $A.$
This excludes the situation when elements of $A$ decay fast, but $B=A^{-1}$ is not bounded, e.g.,
\begin{equation}\nonumber
 A = \begin{bmatrix} 
       1 & & & & \\
       1 & 1 & & & \\
       0 & 1& 1 & & \\
       0 & 0 & 1 & 1 & \\
       \cdot & \cdot & \cdot & \cdot & \cdot  \\
     \end{bmatrix}, \qquad 
 B = \begin{bmatrix} 
       \phantom{-}1 & & & & \\
       -1 & \phantom{-}1 & & & \\
       \phantom{-}1 & -1& \phantom{-}1 & & \\
       -1 & \phantom{-}1 & -1 & \phantom{-}1 & \\
       \cdot & \cdot & \cdot & \cdot & \cdot  \\
     \end{bmatrix}.
\end{equation}

The third case $a_k\to a_* > 0$ was considered under the monotonicity condition $a_0 \geq a_1 \geq a_2 \geq \ldots$ by Vecchio.
The upper bound for the series $\sum_{k=0}^\infty |b_k|$ was established in~\cite{ttoep-bound-2003} and improved later in~\cite{ttoep-bound-2005}. 
It follows that $b_k \to 0$ and the inverse matrix $B=A^{-1}$ belongs to the first class.

Relatively little is known about the second case --- the ``slow decay'' of matrix elements.
The results of Jaffard do not cover this case.
Vecchio mentioned in~\cite{ttoep-bound-2003} that partial sums $u_k = \sum_{j=0}^k b_j$ cannot form the converging series.
The authors of~\cite{ttoep-bound-2005} established the upper bound for $\sum_{j=0}^k |b_j|,$ which grows linearly with $k.$
However, these results do not say much about the properties of $\{b_k\}$ in the limit. 
In this paper we consider this case  and provide new results to fill the gap in the existing literature. 

The paper is organised as follows.
Definitions are introduced in Section~\ref{DEF}.
In Section~\ref{DEC} under the monotonicity condition $a_{k-1} \geq a_k,$ $k\geq1$ we prove that $u_k \to 0$ and therefore $b_k \to 0.$ 
In Section~\ref{RATE} we describe in more quantitative terms how slowly  $u_k$ decays.
In Section~\ref{CONVEX} we prove that for a matrix with slow log--convex decay the inverse matrix is bounded.
In Section~\ref{NUM} we present a numerical example, which illustrates the results we obtained.

\section{Preliminaries and definitions} \label{DEF}
\begin{definition}[\cite{vecchio-sum-2001}]
The \emph{fundamental matrix} $\{u_k\}$ of a sequence $\{a_k\}$ is defined as follows
       $u_{-1}=0,$ $u_k=\sum_{j=0}^k b_j$ for $k\geq0,$ where $\{b_j\}$ defines the inverse matrix.
\end{definition}
The fundamental matrix generates the elements of the inverse matrix as $b_j=u_j-u_{j-1},$ $j\geq0.$
The properties of the fundamental matrix, e.g. limit and summability, allow us to study properties of the inverse matrix.
The following elementary statements can be found in, e.g.,~\cite{vecchio-sum-2001}.
\begin{statement} 
In the definitions made above, the following hold:
  \begin{subequations}
   \begin{align}
   \sum_{j=0}^k a_j b_{k-j} = \sum_{j=0}^k a_{k-j} b_j = 0              & \qquad\mbox{for}\quad k \geq 1; \label{eq:ab}\\
   \sum_{j=0}^k a_j u_{k-j} = \sum_{j=0}^k a_{k-j} u_j = 1              & \qquad\mbox{for}\quad k \geq 0; \label{eq:au}\\
   %\sum_{j=0}^k b_j v_{k-j} = \sum_{j=0}^k b_{k-j} v_j = 1              & \qquad\mbox{for}\quad k \geq 0; \label{eq:bv}\\
   %b_k = \sum_{j=0}^{k-1} b_{j} d_{k-j} = \sum_{j=1}^{k} b_{k-j} d_{j} & \qquad\mbox{for}\quad k \geq 2; \label{eq:bb}\\
    u_k = \sum_{j=0}^{k-1} u_{j} d_{k-j} = \sum_{j=1}^{k} u_{k-j} d_{j} & \qquad\mbox{for}\quad k \geq 1, \label{eq:uu}
    \end{align}
  \end{subequations}
 where $d_k=a_{k-1}-a_k$ for $k \geq 1.$ 
\end{statement}
\begin{proof}
Consider the non-diagonal entries of $AB=I$ to prove~\eqref{eq:ab}.
Summation over the $k'$ leading rows of this linear system gives~\eqref{eq:au} as follows
\begin{equation}\nonumber
 1 = \sum_{k=0}^{k'} \sum_{j=0}^k a_j b_{k-j} =  \sum_{k=0}^{k'} \sum_{j=0}^{k'} a_j b_{k-j} = \sum_{j=0}^{k'} a_j \sum_{k=0}^{k'} b_{k-j} = \sum_{j=0}^{k'} a_j u_{k'-j},
\end{equation}
where we set $b_k = 0$ for $k < 0.$ 
%Similarly, we obtain~\eqref{eq:bv}.
%From~\eqref{eq:ab} obtain~\eqref{eq:bb}  as follows
%\begin{equation}\nonumber
%  a_0b_p  = 0-\sum_{k=0}^{p-1} a_{p-k} b_{k} =  \sum_{k=0}^{p-1} a_{p-1-k} b_{k} - \sum_{k=0}^{p-1} a_{p-k} b_{k} = \sum_{k=0}^{p-1} d_{p-k} c_k.
%\end{equation}
From~\eqref{eq:au} the reccurence relation~\eqref{eq:uu} is written as follows
\begin{equation}\nonumber
  a_0u_k  = 1 - \sum_{j=0}^{k-1} a_{k-j} u_j = \sum_{j=0}^{k-1} a_{k-1-j}u_j - \sum_{j=0}^{k-1} a_{k-j} u_j = \sum_{j=0}^{k-1} d_{k-j} u_j.
\end{equation}
\end{proof}

\section{The decay of elements of the fundamental matrix} \label{DEC}
In this section we assume that $\{a_k\}$ decays monotonically. 
This leads to the following nice statement, see, e.g.~\cite{vecchio-sum-2001}.
\begin{statement}
 If $a_k\geq0$ and $d_k = a_{k-1} - a_k \geq 0$ for all $k\geq 1,$ then 
\begin{equation} \label{eq:u}
  0 \leq u_k \leq 1 \qquad\mbox{for}\quad k \geq 0.
\end{equation}  
\end{statement}
\begin{proof}
 The statement can be proved by an inductive argument.
 Since $u_0=b_0=1,$ the base of induction holds. 
 Then, if $0 \leq u_j \leq 1$ for $1 \leq j \leq k-1,$ we use~\eqref{eq:uu} and we write
 \begin{equation}\nonumber
  \begin{split}
   u_k & = \sum_{j=1}^{k} u_{k-j} d_{j} \geq 0, \\
   u_k & = \sum_{j=1}^{k} u_{k-j} d_{j} \leq \sum_{j=1}^k d_j = 1-a_k \leq 1.
  \end{split} 
 \end{equation}
\end{proof}

We observe that for a triangular Toeplitz matrix $A$ with monotone slow decay the elements of the inverse matrix $B=A^{-1}$ also decay to zero.
\begin{theorem}\label{thm1}
 If $1=a_0 > a_1 \geq a_2 \geq \ldots \geq a_{k-1} \geq a_k \geq \ldots$ and $a_k\to0,$ but the series $\sum_{k=0}^{\infty} a_k$ diverges, then $u_k \to 0.$ 
\end{theorem}
\begin{proof}
Consider all convergent subsequences $\{u_{k_t}\}$ of the sequence $\{u_k\}$ and let
$$
u_* = \max_{\{u_{k_t}\}} \lim_{t\to\infty} u_{k_t}.
$$
By~\eqref{eq:u}, $0 \leq u_* \leq 1.$ 
Suppose that $u_*>0.$
Since $\sum_{k=0}^\infty a_k$ is divergent, we can choose $N$ such that
$$ 
\sum_{k=0}^N a_k > \frac{2}{u_*}
$$ 
and arbitrary small $\eps$ such that 
$$
\eps <  \frac{c-1}{c^N-1} \: \frac{u_*}{2}, \qquad \mbox{where}\quad c=\frac{1}{1-a_1}.
$$
Denote by $\{j_t\}_{t=0}^\infty$ the subsequence of indices for which $u_{j_t} > u_*-\eps.$
If the step sizes of $\{j_t\}$ are bounded (see Fig.~\ref{fig1}, left), i.e.,
$$
\exists h \: \forall t\geq 0: \quad j_{t+1} - j_t \leq h,
$$
then for some sufficiently large $T$ the following inequality holds
$$
\sum_{j=0}^{j_T} a_{j_T-j} u_j \geq 
\sum_{t=0}^T a_{j_T-j_t} u_{j_t} \geq
(u_*-\eps) \sum_{t=0}^{T} a_{j_T-j_t} \geq
(u_*-\eps) \sum_{t=0}^{T} a_{ht} \geq
\frac{u_*-\eps}{h} \sum_{t=0}^{T} a_t > 1.
$$
The contradiction with~\eqref{eq:au} shows that the step sizes of $\{j_t\}$ are not bounded (see Fig.~\ref{fig1}, right). 

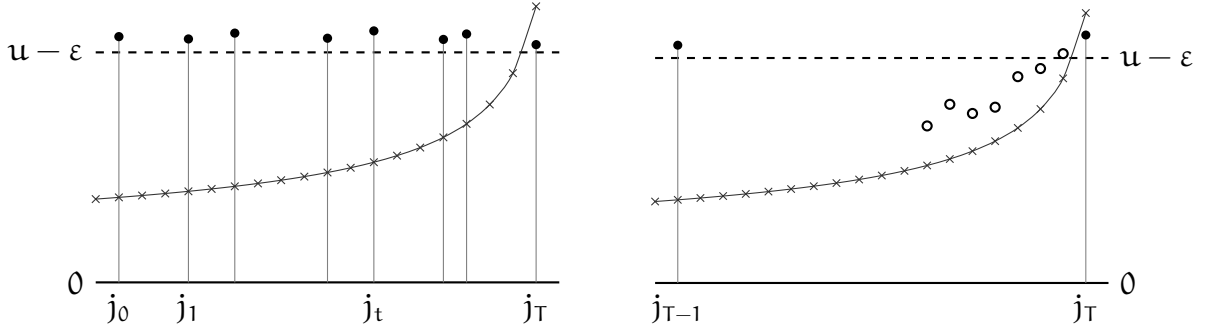
\begin{figure}[t]
 \begin{center}
  \resizebox{.47\textwidth}{!}{\begin{tikzpicture}[x=3mm,y=3mm]
 \draw[thick]         (0,0)  node [anchor=east] {$0$}      -- (20,0) ;
 \draw[dashed,thick]  (0,10) node [anchor=east] {$u-\eps$} -- (20,10);
 \foreach \x in {1,4,6,10,12,15,16,19}
  {
  \draw[gray] (\x,0) -- (\x,10.1+rnd) coordinate (c);
  \fill[black,thick] (c) circle (.6mm);
  }
  \draw ( 1,0) node [anchor=north] {$j_0$};
  \draw ( 4,0) node [anchor=north] {$j_1$};
  \draw (12,0) node [anchor=north] {$j_t$};
  \draw (19,0) node [anchor=north] {$j_T$};
  \draw[color=black!80] plot[domain=0:19,samples=20,mark=x,smooth] (\x,{12/(20-\x)^.4))});
\end{tikzpicture}
 }
  \hfill
  \resizebox{.47\textwidth}{!}{\begin{tikzpicture}[x=3mm,y=3mm]
 \draw[thick]         (0,0)   -- (20,0)  node [anchor=west] {$0$}      ;
 \draw[dashed,thick]  (0,10)  -- (20,10) node [anchor=west] {$u-\eps$} ;
 \foreach \x in {1,19}
  {
  \draw[gray] (\x,0) -- (\x,10.1+rnd) coordinate (c);
  \fill[black] (c) circle (.6mm);
  }
  \foreach \x in {12,13,14,15,16,17,18}
   \draw[black,thick] (\x,\x/2 + 1.5*rnd) circle (.6mm);

  \draw ( 1,0) node [anchor=north] {$j_{T-1}$};
  \draw (19,0) node [anchor=north] {$j_{T}$};
  \draw[color=black!80] plot[domain=0:19,samples=20,mark=x,smooth] (\x,{12/(20-\x)^.4))});
\end{tikzpicture}
 }
 \end{center} 
\caption{Illustration of the Proof of Theorem~\ref{thm1}. 
Marks: $\circ$ --- sequence $u_j,$ $\bullet$ --- subsequence $u_{j_t},$ $\times$ --- sequence $a_{j_T-j}.$
} \label{fig1}
\end{figure}

Choose $M$ such that $a_M \leq \eps,$ and $T$ such that $j_T - j_{T-1} \geq M + N.$
For $j_{T-1} < j < j_T$ all elements $u_j < u_*-\eps,$ since none of them belongs to $\{u_{j_t}\}.$
Set $k=j_T,$ use~\eqref{eq:uu} and write the following inequality
\begin{equation}\label{eq:uniform}
 \begin{split}
  u_*-\eps \leq u_k & = d_1 u_{k-1} + \sum_{j=2}^{M-1} d_j u_{k-j} + \sum_{j=M}^k d_j u_{k-j} \\
                  & \leq d_1 u_{k-1} + \sum_{j=2}^{M-1} d_j (u_*-\eps) + \sum_{j=M}^k d_j \\
                  %& \leq d_1 u_{k-1} + (a_1 - a_M) (u_* - \eps) + (a_M - a_N) \\
                  & \leq (1-a_1) u_{k-1} + a_1 (u_* - \eps) + \eps.
 \end{split}
\end{equation}
This shows that  
$$
u_{k-1} \geq u_* -  \left(1 + \frac{1}{1-a_1} \right)\eps = u_* - c_1 \eps, \qquad c_1 = 1 + c = \frac{c^2-1}{c-1}.
$$
Assume now that for $j<N,$  $u_{k-j+1} \geq u_* - c_{j-1} \eps$ holds with $c_{j-1} = \frac{c^j-1}{c-1}.$ 
To prove the induction step, write (similarly to~\eqref{eq:uniform}) the following inequality
\begin{equation}\nonumber
u_{k-j} \geq u_* - \left(1+ \frac{c_{j-1}}{1-a_1}\right)\eps = u_* - c_j \eps, \qquad c_j = 1 + \frac{c^j-1}{c-1} c= \frac{c^{j+1}-1}{c-1}.
\end{equation}
Using the assumption on $\eps$ we conclude that
\begin{equation}\nonumber
u_{k-j} \geq u_* - c_j \eps > \frac{u_*}{2} \qquad \mbox{for}\quad j=0,\ldots,N-1.
\end{equation}
Now we are ready to show the contradiction with~\eqref{eq:au}.
Indeed, for $k=j_T$ the following holds:
$$
\sum_{j=0}^k a_j u_{k-j} \geq \sum_{j=0}^N a_j u_{k-j} > \frac{u_*}{2} \sum_{j=0}^N a_j > 1.
$$
The contradiction proves $u_* = 0,$ and therefore $\exists \lim_{k\to\infty} u_k = 0.$
\end{proof}
\begin{corollary} 
 Under the conditions of Theorem~\ref{thm1}, $\exists\lim_{k\to\infty}b_k = 0.$ 
\end{corollary}
\begin{remark}
The requirement $a_1 < 1$ in Theorem~\ref{thm1} is technical and can be relaxed. 
\end{remark}
\begin{proof}
Consider the minimal index $l$ such that $a_l < 1.$ 
Define $N$ such that $\sum_{k=0}^N a_{lk} > 2/u_*,$  $c = 1/(1-a_l)$ and $\eps$ and $M$ in the same way as in the proof of the Theorem.
Choose $T$ such that $j_T-j_{T-1} > M + lN,$ set $k=j_T$ and substitute~\eqref{eq:uniform} by
\begin{equation}\nonumber
  u_*-\eps \leq u_k  = d_l u_{k-l} + \sum_{j=l+1}^{M-1} d_j u_{k-j} + \sum_{j=M}^k d_j u_{k-j} \\
                   \leq (1-a_l) u_{k-l} + a_l (u_* - \eps) + \eps.
\end{equation}
This gives $u_{k-l} > u_* - c_1\eps$ and $u_{k-jl} > u_* - c_j\eps$ in the sequel for $j=0,\ldots,N-1.$ 
We have the same contradiction as in the proof of Theorem~\ref{thm1}.
\end{proof}

\section{Summability of the fundamental matrix in $p$--norms} \label{RATE}
In~\cite{ttoep-bound-2003} it is shown that under the conditions of Theorem~\ref{thm1} the series of the fundamental matrix $\sum_{k=0}^\infty u_k$ is not convergent.
In spite of the result of Theorem~\ref{thm1} we have $u_k \to 0,$ $\sum_{k=0}^\infty u_k = \infty,$ i.e., the sequence $u=\{u_k\}$ has slow decay. 

Given a sequence $a=\{a_k\}_{k=0}^\infty$ with slow decay, we sometimes can choose a power $p$ such that $\sum_{k=0}^\infty |a_k|^p < \infty.$ 
A notable example is a harmonic series $\sum_{k=1}^\infty 1/k$ which is divergent, but over-harmonic series $\sum_{k=1}^\infty 1/k^p$ converges for any $p>1.$

A quantitative measure of divergence for a sequence can be given in terms of its summability in $p$--norm.
\begin{definition}
The $p$--norm $\|a\|_p$ of an infinite sequence $\{a_k\}_{k=0}^\infty$ is defined by
$$
\|a\|_p^p = \sum_{k=0}^\infty |a_k|^p, \qquad \|a\|_\infty = \sup_k |a_k|.
$$
\end{definition}
The definition satisfies the axioms of a norm for $p\geq1.$
\begin{definition}
For $p\geq1$ we define by $l^p$ the space of sequences $a$ with $\|a\|_p<\infty.$
\end{definition}
Since $\|a\|_p \geq \|a\|_q$ for $1 \leq p \leq q,$ the embedding $l^q \subset l^p$ holds for $q \geq p.$
For sequences with slow decay the following definition makes sense.
\begin{definition}\label{rate}
For a sequence $a=\{a_k\}_{k=0}^\infty$ such that $\lim_{k\to\infty}a_k=0$ and $\sum_{k=0}^\infty |a_k|^p=\infty,$ find $p\geq1$ such that $a\notin l^p$ but $a\in l^q$ for all $q>p.$ 
The value $0 \leq 1/p \leq 1$ will be referred to as the \emph{decay rate} of $a.$
\end{definition}
\begin{example}
The harmonic series sequence $a_k=1/(k+1)$ has a decay rate $1.$
\end{example}
\begin{example}
The sub-harmonic series sequence $a_k = (1+k)^{-\alpha},$ $0 < \alpha \leq 1,$  has the decay rate $\alpha.$
\end{example}
\begin{remark}
The reverse of the last example is not true: if a sequence has decay rate $\alpha,$ we cannot claim that $a_k \leq c (1+k)^{-\alpha}$ for some $c>1.$
\end{remark}

The analysis of the decay rate of the fundamental matrix is based on Young's convolution theorem~\cite{young-conv}.
It is one of the most basic resuls in harmonic analysis, which plays an important role, e.g., in PDE theory.
\begin{theorem}[Young's inequality for discrete convolution] \label{young}
 Let $z_k = \sum_{j=0}^k x_j y_{k-j}$ for $k\geq0.$ 
 For $1 \leq p,q,r \leq \infty$ such that 
 $$
 1 + \frac1r = \frac1p + \frac1q 
 $$
 it follows that
 \begin{equation}\nonumber %\label{eq:young}
  \|z\|_r \leq \|x\|_p  \|y\|_q.
 \end{equation}
\end{theorem}
The discrete version of this theorem is not common in the literature and we provide the proof in our Appendix.
Using this inequality, we can estimate the decay rate of the fundamental matrix.
\begin{theorem}\label{decay}
Consider a triangular Toeplitz matrix generated by a nonnegative slowly decaying sequence 
$$a=\{a_k\}_{k=0}^\infty, \qquad a_k\geq0,\qquad \lim_{k\to\infty} a_k = 0, \qquad \sum_{k=0}^\infty a_k=\infty.$$ 
If $a$ has decay rate $0 \leq \alpha \leq 1,$ the fundamental matrix has decay $\upsilon \leq 1 - \alpha.$
\end{theorem}
\begin{proof}
The result of Vecchio~\cite{ttoep-bound-2003} proves $\upsilon < 1.$
Suppose that $\alpha + \upsilon > 1,$ then according to the Definition~\ref{rate} 
$$ 
\exists p,q\geq1, \quad \frac1p + \frac1q > 1, \qquad\mbox{such that} \quad \|a\|_p < \infty,  \quad \|u\|_q < \infty.
$$
By Young's inequality for the sequence $z_k = \sum_{j=0}^k a_j c_{k-j}$ there is $1 < r < \infty$ such that 
$$\|z\|_r \leq \|a\|_p \|c\|_q < \infty.$$
However, by~\eqref{eq:au}, $z_k=1$ for all $k\geq 0$ and $\|z\|_r = \infty$ for all $r < \infty.$
The conclusion of the theorem follows by contradiction.
\end{proof}

\section{Inverse and fundamental matrices in the log--convex case} \label{CONVEX}
Following~\cite{kingman-1961}, a function $f(x)$ is log--convex (or \emph{superconvex}) if $\log f(x)$ is convex.
A similar notion is defined for sequences as follows.
\begin{definition}
A sequence $a=\{a_k\}_{k=0}^\infty, a_k\geq0$ is called \emph{log--convex} if $a_k\geq0$ and  
$$a_k^2\leq a_{k-1}a_{k+1} \qquad\mbox{for}\quad k\geq1.$$
\end{definition}
Log--convex functions and sequences are often used to study densities and discrete distributuions in probability. 
\begin{remark}\label{pos}
 If a log--convex sequence satisfies $a_0 > 0$ and $a_1 > 0$ then all other elements are also positive.
\end{remark}

\begin{theorem}\label{thm4}
 For a triangular Toeplitz matrix $A$ defined by a nonegative log--convex sequence, the inverse matrix $B=A^{-1}$ has elements $b_k\leq0$ for $k\geq1.$
\end{theorem}
\begin{proof}
The invertibility of $A$ requires $a_0\neq0,$ and we assume w.l.o.g. $a_0=1.$
If $a_1=0$ and $a$ is log--convex then all further elements of the sequence are zeroes.
Indeed, $a_2^2 \leq a_1a_3=0$ hence $a_2=0$ and so on.
Such a sequence defines a unit matrix $A=I$ with $B=A^{-1}=I.$  
The statement of Theorem holds for this trivial case.

If $a_1 > 0$ then all further elements of the log--convex sequence are also strictly positive.
To prove this, apply $a_{k+1} \geq a_k^2 / a_{k-1} > 0$ recursively for $k\geq2.$
Therefore all elements of $a$ are non--zeroes and we can divide by them.

Since $a_0=1,$ the inversion equation gives $b_1 = -a_1 \leq 0.$
Assume that $b_j \leq 0$ for $j=1,\ldots,k$ and prove the same for $b_{k+1}.$
From~\eqref{eq:ab} it follows that $a_k = -\sum_{j=1}^{k} a_{k-j} b_j$ and for $k\geq1$ it follows that
\begin{equation}\nonumber
 \begin{split}
  -\frac{a_k}{a_{k-1}} & = \sum_{j=1}^k \frac{a_{k-j}}{a_{k-1}} b_j, \\
  -\frac{a_{k+1}}{a_k} & = \sum_{j=1}^{k+1} \frac{a_{k+1-j}}{a_k} b_j = \sum_{j=1}^{k} \frac{a_{k+1-j}}{a_k} b_j + \frac{b_{k+1}}{a_k}.
 \end{split}
\end{equation}
Substracting, we obtain
\begin{equation}\nonumber
  \frac{a_k}{a_{k-1}} -\frac{a_{k+1}}{a_k} = \sum_{j=1}^{k} \left(\frac{a_{k+1-j}}{a_k} - \frac{a_{k-j}}{a_{k-1}} \right)  b_j + \frac{b_{k+1}}{a_k},
\end{equation}
where the left-hand side is non-positive since $a$ is log--convex.
Similarly, each round bracket in the right--hand side has the same sign as
\begin{equation}\nonumber
  \frac{a_{k+1-j}}{a_{k-j}} - \frac{a_k}{a_{k-1}} 
      =  \left( \frac{a_{k+1-j}}{a_{k-j}} - \frac{a_{k+2-j}}{a_{k+1-j}} \right) 
      + \left( \frac{a_{k+2-j}}{a_{k+1-j}} - \frac{a_{k+3-j}}{a_{k+2-j}} \right) 
      + \ldots + \left( \frac{a_{k-1}}{a_{k-2}} - \frac{a_k}{a_{k-1}} \right) \leq 0,
\end{equation}
where each term is non-positive due to the log--convexity of $a.$
Finally,
\begin{equation}\nonumber
 \frac{b_{k+1}}{a_k} = \left(\frac{a_k}{a_{k-1}} -\frac{a_{k+1}}{a_k}\right) - \sum_{j=1}^{k} \left(\frac{a_{k+1-j}}{a_k} - \frac{a_{k-j}}{a_{k-1}} \right)  b_j \leq 0.
\end{equation}
Each round bracket in the right--hand side is non-positive, and all $b_j\leq0,$ $j=1,\ldots,k,$ are non-positive by the assumption of our recursion.
It follows that $b_{k+1}\leq0,$ and the theorem is proved by recursion.
\end{proof}

\begin{theorem}\label{final}
If a triangular Toeplitz matrix saisfies both the conditions of Thm.~\ref{thm1} and Thm.~\ref{thm4}, i.e. has slow log--convex decay, then $B=A^{-1}$ has fast decay, and $\|B\|_1 = \sum_{k=0}^\infty |b_k| = 2.$   
\end{theorem}
\begin{proof}
For matrices with slow decay the result of Thm.~\ref{thm1} gives $0=\lim_{k\to\infty} u_k = 1 + \sum_{k=1}^{\infty} b_k.$ 
For matrices with log--convex decay Thm.~\ref{thm4} proves $b_k \leq 0$ for $k\geq1.$
Taking these conditions together we have 
$
\sum_{k=0}^\infty |b_k| = 1 - \sum_{k=0}^\infty b_k = 2,
$
which completes the proof.
\end{proof}

\section{Numerical example} \label{NUM}
\begin{figure}[t]
 \begin{center}
      \includegraphics[width=.495\textwidth]{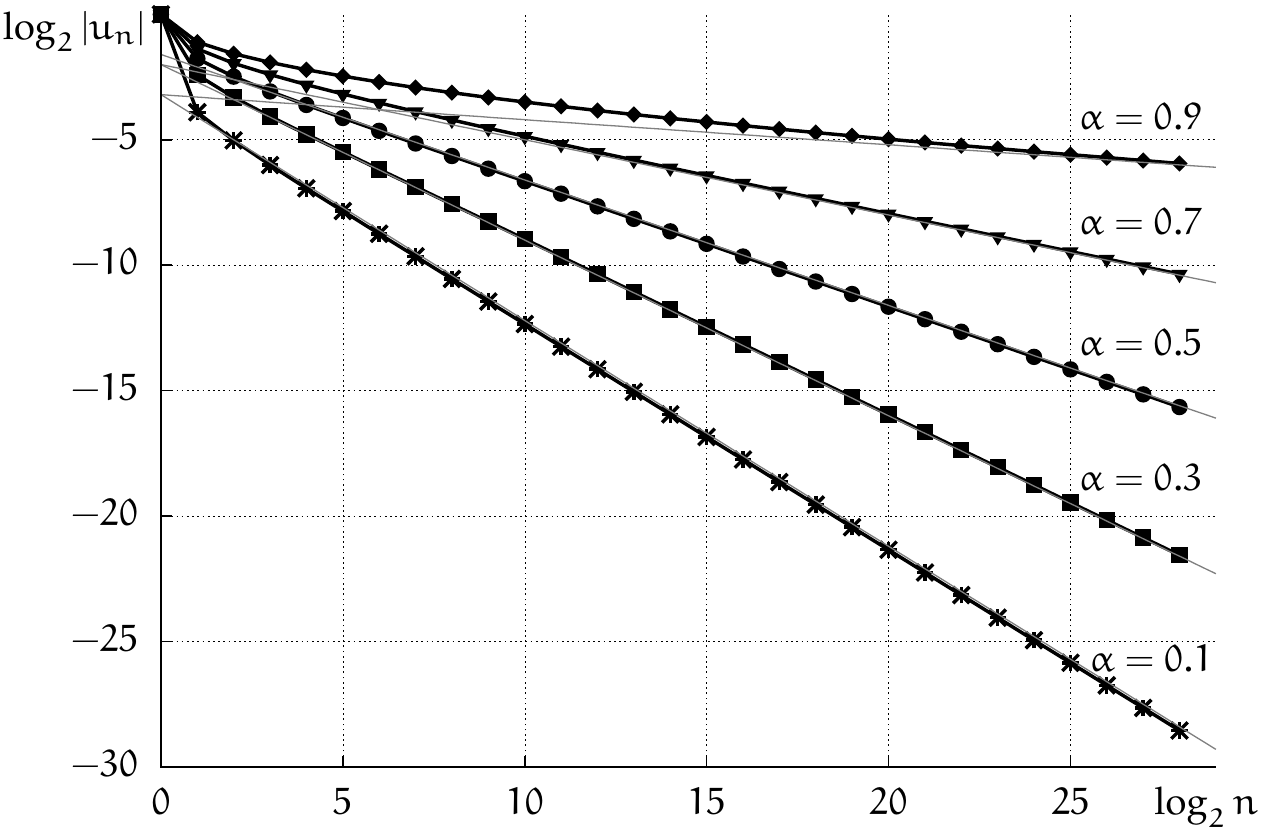} \hfil
      \includegraphics[width=.495\textwidth]{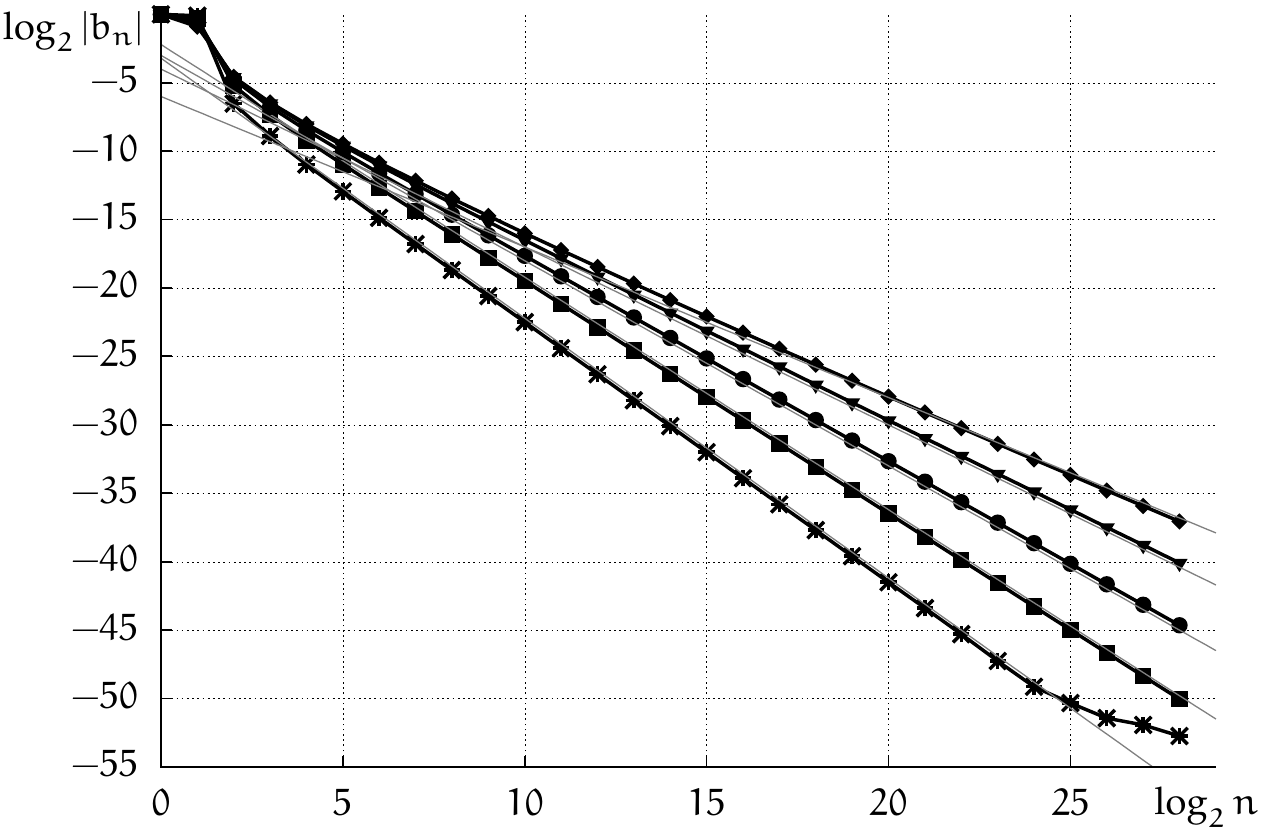}
 \end{center}
 \caption{Decay of the fundamental matrix (left) and inverse matrix (right) for the triangular Toeplitz matrix with elements $a_k=(1+k)^{-\alpha}$ for different $\alpha < 1.$} \label{fig:a}
\end{figure}
We consider the triangular Toeplitz matrix generated by a sequence $a_k=(1+k)^{-\alpha},$ which is log--convex and has slow decay for $\alpha<1.$ 
For different values of $\alpha$ we have computed the inverse using the divide-and-conquer algorithm~\cite{morf-dc-1980,commenges-dc-1984} for very large matrices.
On Fig.~\ref{fig:a} we show the decay of elements of the inverse and the fundamental matrix for different $\alpha.$
We observe that the rate of decay $\upsilon$ for the fundamental matrix behaves in accordance with the result of Thm.~\ref{decay}, i.e. $\upsilon = 1 - \alpha.$ 
Note that the example seems to provide a sharp bound for the inequality in Thm.~\ref{decay} but we do not have a theoretical proof of this fact yet.

Since in the example considered here, the matrix has log--convex decay, the fundamental matrix $u_k$ decays monotonically. 
It is no surprise that the elements of the inverse matrix, which behave like a numerical derivative of $u_k,$ have the decay rate $\beta=1+\upsilon = 2-\alpha,$ which is clearly observed in Fig.~\ref{fig:a}.

\section{Conclusion}
For the triangular Toeplitz matrices with slow decay we have established new results on the decay of the inverse and the fundamental matrix.
A particularly interesting case is established by Thm.~\ref{final}, in which we considered a matrix with slow log--convex decay and proved that the inverse matrix is bounded. 
The proposed results extend the classical analysis of Jaffard~\cite{jaffard-1990} and the results of Veccio et
al~\cite{vecchio-sum-2001,ttoep-bound-2003,ttoep-bound-2005}.

The proposed results may be used to prove the stability of numerical schemes for convolutional Volterra equations of the first
kind~\cite{gripenberg-volterra1-1980}, which are less studied than the equations of the second kind~\cite{lubich-volterra-1986,blanck-1995,blanck-1996,brunner-2004}.

\newpage
\appendix
\section{Young's inequality for discrete convolutions}
Here we provide the proof of Young's convolution theorem~\ref{young} for sequences.
We start from several lemmas.
\begin{lemma}[Young's inequality for products~\cite{young-prod}]
For non-negative $x, y$ and $p, q\geq 1$ such that $1/p + 1/q = 1$,
 \begin{equation}\label{eq:prod}
  xy \leq \frac{x^p}{p} + \frac{y^q}{q}.
 \end{equation}
\end{lemma}
\begin{lemma}[H\"older's inequality]
 \begin{equation}\label{eq:hol}
 \|xy\|_1 \leq \|x\|_p \|y\|_q \qquad\mbox{for}\quad 1 = \frac1p + \frac1q, \quad p,q \geq 1,
 \end{equation}
 where $z=xy$ denotes the elementwise product of sequences $x$ and $y,$ i.e., $z_j=x_jy_j,$ $j\geq0.$
\end{lemma}
\begin{proof}
 If $\|x\|_p=0$ or $\|y_q\|=0$ then $xy=0$ and the result is trivial.
 For non-zero $x$ and $y$ w.l.o.g. we set $\|x\|_p=\|y\|_q=1.$ 
 Then using~\eqref{eq:prod} we write
 $$
 \sum_{j=0}^\infty \left|x_jy_j\right| 
       \leq \sum_{j=0}^k \left( \frac{|x_j|^p}{p} + \frac{|y_j|^q}{q} \right) 
       \leq \frac{\|x\|^p}{p} + \frac{\|y\|^q}{q} 
       = \frac1p + \frac1q = 1,
 $$
 which proves the inequality.
\end{proof}
\begin{lemma}[Generalized H\"older's inequality]
If $\sum_{k=1}^m 1/p_k = 1/r$ for $p>0$ and $0<r<\infty,$ and sequences $x_k\in\l_k,$ then
\begin{equation}\label{eq:Hol}
 \|x_1 x_2 \ldots x_m\|_r \leq \|x_1\|_{p_1} \|x_2\|_{p_2}  \ldots \|x_m\|_{p_m}.
\end{equation}
\end{lemma}
\begin{proof}
For $m=1$ the result is obvious.
Suppose the result holds for $m-1$ sequences  $x_1 \ldots x_{m-1}$, we shall prove it for $m.$
If $p_m=\infty,$ the result follows by pulling out the supremum of  $x_m$ and using the induction hypothesis.
For $p_m<\infty$ consider $1/p = 1-r/p_n$ and $1/q=r/p_n,$ which form a H\"older pair $1/p+1/q=1.$
Using~\eqref{eq:hol}, we write
\begin{equation}\nonumber
 \begin{split}
  \left\| |x_1 \ldots x_{m-1}|^r   |x_m|^r\right\|_1 & \leq \left\| |x_1 \ldots x_{m-1}|^r \right\|_p \left\| |x_m|^r\right\|_q, \\
 \left\| x_1 \ldots x_{m-1} x_m\right\|_r & \leq \left\| x_1 \ldots x_{m-1} \right\|_{pr} \left\| x_m\right\|_{qr}, 
 \end{split}
\end{equation}
and obtain the result by induction since $qr=p_m$ and $\sum_{k=1}^{m-1}1/p_k = 1/r-1/p_m = 1/(pr).$
\end{proof}

%Now we are ready to prove the convolution inequality.
\begin{proof}[Theorem~\ref{young}]
We start from the following simple cases, assuming w.l.o.g. that  $p\leq q.$
 
 (A) $p=q=r=1.$ It is enough to write
 $$
 \|z\|_1 = \sum_{k=0}^\infty \left|z_k\right|
   \leq \sum_{k=0}^\infty \sum_{j=0}^k \left|x_j y_{k-j}\right| 
   = \sum_{j=0}^\infty \left|x_j\right| \: \sum_{k=0}^\infty \left| y_{k-j} \right| 
   = \|x\|_1 \|y\|_1.
 $$
 
 (B) $r=\infty,$ $1/p + 1/q = 1.$ 
Since
 $ 
 \left|z_k\right|  \leq \sum_{j=0}^k \left| x_j  y_{k-j} \right|  =\sum_{j=0}^\infty \left| x_j  y_{k-j} \right|, 
 $
 the result follows immediately from H\"older's inequality~\eqref{eq:hol}
applied to sequences $\hat x=\{|x_j|\}_{j=0}^\infty$ and $\hat y
=\{|y_{k-j}|\}_{j=0}^\infty.$
 Here and later we assume $y_j = 0$ for $j < 0.$
  
 (C) $p=1,$ $1<q=r<\infty.$ Consider $q'>1$ such that $1/q + 1/q' = 1.$ For any $k$ 
 by H\"older's inequality we have 
 \begin{equation}\nonumber
  \begin{split}
   |z_k| 
     & = \left| \sum_{j=0}^k x_j y_{k-j} \right|
       \leq \sum_{j=0}^k |x_j y_{k-j}|
       =  \sum_{j=0}^\infty |x_j y_{k-j}|
       =  \sum_{j=0}^\infty \left( |y_{k-j}|^{1/q'} \right) \left( |x_j|  |y_{k-j}|^{1/q} \right) \\
      & \leq \left( \sum_{j=0}^\infty |y_{k-j}| \right)^{1/q'} 
            \left( \sum_{j=0}^\infty |x_j|^q \: |y_{k-j}| \right)^{1/q}
      \leq \|y\|_1^{1/q'} \left( \sum_{j=0}^\infty |x_j|^q \: |y_{k-j}| \right)^{1/q},
   \end{split}
 \end{equation} 
$$ 
 \|z\|_q^q 
    = \sum_{k=0}^\infty |z_k|^q 
    \leq \|y\|_1^{q/q'} \sum_{k=0}^\infty \sum_{j=0}^\infty |x_j|^q \: |y_{k-j}| 
    = \|y\|_1^{1+q/q'} \|x\|^q_q
    = \|y\|_1^q \|x\|^q_q,
$$
which completes the proof of the case (C).
 
Now we deal with the final case $1 < p \leq q < r < \infty.$
For each $k$ consider again the sequences $\hat x$ and $\hat y,$ write
$$
|x_j y_{k-j}| = \left( |x_j|^p  |y_{k-j}|^q \right)^{1/r} \: |x_j |^{1-p/r}  |y_{k-j}|^{1-q/r}, 
$$
and apply the generalized H\"older inequality~\eqref{eq:Hol} with 
$$
p_1=r, \quad p_2=\frac{p}{1-p/r}, \quad p_3=\frac{q}{1-1/r}, \qquad 
    \frac{1}{p_1} + \frac{1}{p_2}+ \frac{1}{p_3}=\frac1r + \frac1p-\frac1r+\frac1q-\frac1r = 1.
$$
We have
\begin{equation}\nonumber
 \begin{split}
 |z_k| &\leq\sum_{j=0}^k |x_j y_{k-j}|  \leq \left( \sum_{j=0}^k|x_j|^p  |y_{k-j}|^q \right)^{1/r} \|x\|_p^{1-p/r} \|y\|_q^{1-q/r}, \\ 
 \|z\|_r^r &\leq  \|x\|_p^{r-p} \|y\|_q^{r-q} \sum_{k=0}^\infty \sum_{j=0}^\infty|x_j|^p  |y_{k-j}|^q = \|x\|_p^r \|y\|_q^r,
 \end{split}
\end{equation}
which completes the proof.
\end{proof}

%%% BIB
\newpage
%\bibliographystyle{mysiam}
%\bibliographystyle{elsarticle-num}
%\bibliography{our,tensor,fft,algebra,frac,iter}

\end{document}